\documentclass{article}

\usepackage{amsmath, amssymb, amsfonts, amstext, verbatim, amsthm, mathrsfs}
\usepackage{bm}
\numberwithin{equation}{section} %label equation by sections
\usepackage{geometry}
\usepackage{tikz-cd}
\usepackage{tikz} %clashes with color package option "usenames"
\usepackage{enumerate}
\usepackage{enumitem}
\usepackage{appendix}

\usepackage[colorlinks=true,linkcolor=blue,citecolor=blue,urlcolor=blue,citebordercolor={0 0 1},urlbordercolor={0 0 1},linkbordercolor={0 0 1}]{hyperref} %needs to be loaded after most things
\usepackage[nameinlink]{cleveref}

\theoremstyle{plain}
\newtheorem*{thm}{Theorem}
\newtheorem{theorem}{Theorem}[section]
\newtheorem{corollary}{Corollary}[theorem]
\newtheorem{proposition}[theorem]{Proposition}
\newtheorem*{prop}{Proposition}
\newtheorem{lemma}[theorem]{Lemma}
\newtheorem{question}{Question}
\newtheorem*{ques}{Question}

\theoremstyle{definition}
\newtheorem*{de}{Definition--Theorem}
\newtheorem{definition}[theorem]{Definition}
\newtheorem{example}[theorem]{Example}

\theoremstyle{remark}
\newtheorem{remark}{Remark}

\providecommand{\keywords}[1]
{
  \small	
  \textbf{Key Words } #1
  \normalsize
}

\providecommand{\class}[1]
{
  \small	
  \textbf{Mathematics Subject Classification 2020 } #1
  \normalsize
}

\AtEndDocument{%
\bibliographystyle{plain}
\bibliography{./reference.bib}
}
\title{Derived representation schemes with arbitrary coefficients and associative smoothness}
\author{Guanyu Li}
\begin{document}
\maketitle

%%%%%%%%%%%%%%%%%%%%%%%%%%%%%%%%%%%%%%%%%%%%%

\begin{abstract}
We study associative smoothness through representation homology with coefficients in arbitrary finite-dimensional algebras.
We prove that the higher representation homology of a finitely generated formally smooth algebra with arbitrary finite-dimensional coefficients vanishes.
We further show that allowing arbitrary coefficients yields a strictly stronger smoothness test: suitable non-matrix coefficients detect nonsmoothness in situations where matrix coefficients do not.
These results raise the question of whether representation homology with arbitrary coefficients furnishes a homotopical characterization of associative smoothness in the spirit of the derived Kontsevich–Rosenberg principle.
For finite-dimensional algebras, we prove that representation homology with arbitrary coefficients completely characterizes formal smoothness, providing supporting evidence for this philosophy.
\end{abstract}

%%%%%%%%%%%%%%%%%%%%%%%%%%%%%%%%%%%%%%%%%%%%%
\keywords{Derived representation schemes, representation homology, associative smoothness, Kontsevich-Rosenberg principle.}
~\par

\class{Primary: 14A22, 14A30, Secondary: 16S38, 18G90.}

\section{Introduction}

Given an associative algebra $A$ over a field $k$, its representations in a finite-dimensional vector space $V$ are parametrized by an affine scheme $\operatorname{Rep}_V(A)$, called the representation scheme.
Following Kontsevich and Rosenberg (\cite[Section 9]{K93}, \cite{KR00}), one may regard these schemes as commutative approximations to the ``noncommutative scheme" $\mathrm{Spec}~A$, leading to a heuristic principle that geometric properties of $A$ should be reflected in the geometry of $\operatorname{Rep}_V(A)$.
This philosophy extends to the derived setting through the derived representation schemes $\operatorname{DRep}_V(A)$ and their algebraic invariants $\mathrm{HR}_*(A,V)$, called \textbf{representation homology}, which measure the failure of the classical approximation (\cite{BFR14,BKR13}).
A notable realization of the derived Kontsevich-Rosenberg principle is the vanishing theorem of Berest–Felder–Ramadoss: if $A$ is formally smooth under additional finiteness hypotheses, then $\mathrm{HR}_i(A,V)=0$ for $i>0$ (\cite[Theorem 20]{BFR14}).

While representation homology detects nonsmoothness invisible to ordinary representation schemes, the converse of the higher-vanishing theorem nevertheless fails in general.
Indeed, there exist associative algebras that are not formally smooth but whose higher representation homology groups vanish; see \cite[Section 6.2.2]{BFR14}\footnote{In this example, the ordinary representation scheme is itself singular.}.
\emph{This suggests that representation homology with matrix coefficients is not sufficient to detect smoothness in general.}
The guiding idea of this paper is to enlarge the class of coefficients used in representation homology.
A representation of an associative algebra $A$ is an algebra homomorphism $A\to\operatorname{End}_k(V)$ for a finite-dimensional vector space $V$ over $k$.
Motivated by the failure of the converse, we replace $\operatorname{End}_k(V)$ by any finite-dimensional algebra $M$ and consider algebra homomorphisms $A\to M$ which we call \textbf{representations of $A$ with coefficient $M$}.
Such representations are also parameterized by an affine scheme, whose derived enhancement yields a natural generalization of the derived representation scheme.

We first prove that formal smoothness implies vanishing of higher representation homology for all finite-dimensional coefficients.

\begin{thm}\label{Thm}[\Cref{Thm_SmoothnessHigherVanishing}]
Let $M$ be a finite-dimensional $k$-algebra. If $A$ is a finitely generated formally smooth $k$-algebra, then $\mathrm{HR}_i(A,M)$ vanishes for all $i>0$.
\end{thm}

The existing proof of the vanishing theorem (\cite[Theorem 21]{BFR14}) proceeds via a comparison theorem of Ciocan-Fontanine--Kapranov for smooth DG schemes.
This requires a semifree resolution whose underlying DG algebra is locally finite-dimensional, a hypothesis that is independent of the classical notion of formal smoothness.
While such resolutions are expected to exist in many practical situations, verifying the required finiteness conditions can be nontrivial.
Our proof replaces this step by a cotangent-complex argument and Neeman's fiberwise detection theorem, thereby recovering the vanishing theorem under the assumptions that $A$ is finitely generated and formally smooth.

In particular, the theorem applies to examples such as $U\mathfrak{g}$ for a semisimple Lie algebra $\mathfrak{g}$, where establishing the required local finiteness properties of semifree resolutions appears more difficult than verifying formal smoothness itself (see \Cref{Example_UniversalEnveloping}).

We also isolate the key argument underlying \Cref{Thm_SmoothnessHigherVanishing} as a comparison result in the appendix.
It recovers the comparison theorem of Ciocan-Fontanine--Kapranov and may be of separate interest in other contexts of derived algebraic geometry.

\

Although every finite-dimensional algebra embeds into a matrix algebra, derived representation schemes with non-matrix coefficients may detect singularities that are invisible in the classical matrix-valued version.
Indeed, the geometry of the representation scheme $\operatorname{Rep}_M(A)$ depends on the algebra structure of $M$.
By choosing appropriate coefficients, a direct computation shows (see \Cref{Example_QuantumPlane,Example_JordanPlane})
\begin{equation*}
\mathrm{HR}_1(k[x,y]_q,(\mathfrak{n}_3)_+)\neq0\quad\text{and}\quad\mathrm{HR}_1\left(\frac{k\langle x,y\rangle}{\langle yx-xy-x^2\rangle},(\mathfrak{n}_3)_+\right)\neq0.
\end{equation*}
\emph{Thus both the quantum plane (for $q$ not a root of unity) and the Jordan plane can be shown nonsmooth using representation homology with non-matrix coefficients}.

These examples demonstrate that the passage from matrix coefficients to arbitrary finite-dimensional coefficients is not merely a routine extension, but leads to genuinely stronger smoothness-detection results.
It is therefore natural to ask whether the converse holds for arbitrary finite-dimensional coefficients:

\begin{ques}[\Cref{ques}]
Let $A$ be a finitely generated Noetherian associative algebra over $k$ with sufficiently many finite-dimensional representations\footnote{Namely, for any finite-dimensional algebra $M$, there exist algebra homomorphisms $A\to M$. For instance, the Weyl algebra $A_1$ fails this condition.}.
If for \emph{every} finite-dimensional algebra $M$, the higher representation homology groups $\mathrm{HR}_i(A,M)$ vanish for all $i>0$, is $A$ formally smooth?
\end{ques}

The evidence for the question is not purely experimental.
We prove that it has an affirmative answer for finite-dimensional algebras.

\begin{prop}[\Cref{Prop_FinDimAlg}]
A finite-dimensional $k$-algebra $A$ over an algebraically closed field $k$ is formally smooth if and only if $\mathrm{HR}_i(A,M)$ vanishes for all $i>0$ and for every finite-dimensional algebra $M$.
\end{prop}

Derived representation schemes with arbitrary finite-dimensional coefficients therefore provide a strictly stronger smoothness test than the classical matrix-valued derived representation schemes.
For finite-dimensional algebras, they completely characterize associative smoothness.

\subsection{Organization of the paper}

In \Cref{Sec_Pre}, we collect the necessary background on derived representation schemes and representation homology.
In \Cref{Sec_FormalSm}, we provide a new proof that if a finitely generated associative algebra is formally smooth, then its higher representation homology vanishes.
In \Cref{Sec_ExQ}, motivated by several illustrative examples, we  ask whether the generalized derived Kontsevich–Rosenberg principle provides a complete characterization of formal smoothness.
We also show that for finite-dimensional algebras, derived representation schemes with arbitrary finite-dimensional coefficients give a complete characterization of formal smoothness.
In the appendix, we isolate the key argument from the proof of the main theorem and establish a comparison result in the setting of derived algebraic geometry, which recovers a known result of Ciocan-Fontanine--Kapranov.

\subsection{Notations and conventions}

Throughout this paper, $k$ will denote a field of characteristic $0$. Unless specified, the tensor product $\otimes$ will be over $k$.
By a $k$-algebra $A$, we mean an \emph{associative} unital ring $A$ over $k$, not necessarily commutative.
We shall denote by ${\bf Alg}_k$ the category of $k$-algebras, and by ${\bf CommAlg}_k$ the full subcategory of all commutative $k$-algebras.
$A^e:=A^{\mathrm{op}}\otimes A$ is the enveloping algebra.
Similarly, a coalgebra is always coassociative, counital, and over $k$.

We will always use homological convention, namely the differentials are always of degree $-1$.
The categories of DG algebras and commutative DG algebras over $k$ are denoted ${\bf DGA}_k$ and ${\bf CDGA}_k$, respectively.
The full subcategory of ${\bf DGA}_k$ (resp. ${\bf CDGA}_k$) consisting of nonnegatively graded DG-algebras is denoted by ${\bf DGA}_k^{\geq0}$ (resp. ${\bf CDGA}_k^{\geq0}$).
As usual, ${\bf Alg}_k$ will be identified with the full subcategory of ${\bf DGA}_k$ consisting of DG algebras with a single nonzero component in degree $0$.
%For any category $\mathcal{C}$ and an object $X\in\mathcal{C}$, we denote by $\mathcal{C}/X$ the category over $X$, consisting of all morphisms $A\to X$, and $X/\mathcal{C}$ the category under $X$.
%Given any category $\mathcal{C}$, denote by $s\mathcal{C}$ the category of simplicial objects of $\mathcal{C}$, i.e. $s\mathcal{C}:=\mathrm{Funct}(\bf\Delta^{\mathrm{op}},\mathcal{C})$.
${\bf DGA}_k$ (resp. ${\bf CDGA}_k$ when $\operatorname{char}k=0$) carries a natural model structure.
For detailed results about simplicial sets and model categories, we refer to \cite{BKR13,H03,Q67}.

\subsection{Acknowledgements}

The ideas developed in this paper originated in the author's doctoral thesis.
The author is grateful to his advisor Yuri Berest for introducing him to this subject and for many helpful discussions.
The author would also like to express his gratitude to Isaac Goldberg and Ajay Ramadoss, for reading the first draft and offering suggestions.

\section{Preliminaries on derived representation schemes}\label{Sec_Pre}

Representation homology of associative algebras was introduced and studied in \cite{BFPRW17,BFR14,BKR13,BR16} and subsequent works, as a derived generalization of the representation scheme.
In this section, we recall the notion and basic properties of representation homology of algebras with arbitrary finite-dimensional coefficients.

\begin{de}[{\cite[Section A.1]{BFPRW17}}]
Given a finite-dimensional coalgebra $C$, the functor $\mathrm{Hom}_k(C,-):\bm{\mathrm{Alg}}_k\to\bm{\mathrm{Alg}}_k$ admits a left adjoint denoted by $\sqrt[C]{-}:\bm{\mathrm{Alg}}_k\to\bm{\mathrm{Alg}}_k$ following the notations in \cite{BKR13}.
When restricted to the subcategory $\bm{\mathrm{CommAlg}}_k$, the functor $\mathrm{Hom}_k(C,-):\bm{\mathrm{CommAlg}}_k\to\bm{\mathrm{Alg}}_k$ admits an adjunction
\begin{equation}\label{Eq_RepFunctorCoalgCoefficientAdj}
(-)_C:\bm{\mathrm{Alg}}_k\to\bm{\mathrm{CommAlg}}_k:\mathrm{Hom}_k(C,-)
\end{equation}
where $(-)_C:=(\sqrt[C]{-})_{\mathrm{ab}}$.
In this setting, we call $\operatorname{Spec}A_C$ the \textbf{representation scheme} of $A$ with coefficient $C$.

The adjunction (\ref{Eq_RepFunctorCoalgCoefficientAdj}) extends to an adjunction
\begin{equation*}
(-)_C:\bm{\mathrm{DGA}}_k\to\bm{\mathrm{CDGA}}_k:\mathrm{Hom}_k(C,-)
\end{equation*}
which is a Quillen pair.
Hence $(-)_C$ has a left derived functor $\mathbb{L}(-)_C$ computed by $\mathbb{L}(A)_C=(QA)_C$, where $QA\to A$ is a cofibrant replacement of $A$ in $\bm{\mathrm{DGA}}_k$.

Define $\mathrm{HR}_*(A,C):=\mathrm{H}_*(\mathbb{L}(A)_C)$ for a finite-dimensional coalgebra $C$, or $\mathrm{HR}_*(A,M):=\mathrm{H}_*(\mathbb{L}(A)_{M^*})$ for a finite-dimensional algebra $M$, and call the homology groups the \textbf{representation homology of $A$ with coefficient $C$} (or $M$).
We call the associated derived scheme \textbf{derived representation scheme}, and denote it by $\operatorname{DRep}(A,C)$ (or $\operatorname{DRep}(A,M)$).
\end{de}

\begin{proposition}[c.f. {\cite[Theorem 2.5]{BKR13}}]\label{Proposition_ZeroTrunction}
Given any finite-dimensional coalgebra $C$, then for any $A\in{\bf Alg}_k$, $\mathrm{HR}_0(A,C)\cong A_C$.
\end{proposition}

\begin{definition}\label{Def_LocallyFinite}
We say a graded vector space $V_\bullet$ is \textbf{locally of finite dimension} if $V_\bullet$ is concentrated in nonnegative degrees and $\dim V_i<\infty$ for all $i\geq0$.
We say a DG algebra $A$ is \textbf{free of locally finite dimension} if $A=T_k(V_\bullet)$ as graded algebras where $V_\bullet$ is locally of finite dimension.
\end{definition}

\begin{example}\label{Example_FreeAlg}
Let $V$ be a finite-dimensional vector space and let $A=T_k(V)$ be the free $k$-algebra generated by a set $X$ of basis elements of $V$.
By the universal property, for any commutative algebra $B$, we have a canonical isomorphism
\begin{align*}
\operatorname{Hom}_{{\bf CommAlg}_k}((T_k(V))_C,B)&\cong\operatorname{Hom}_{{\bf Alg}_k}(k\langle X\rangle,\operatorname{Hom}_k(C,B))\\
&=\prod_{x\in X}\operatorname{Hom}_k(C,B).
\end{align*}
Note that $\operatorname{Sym}_k(V\otimes C)$ satisfies the property
\begin{equation*}
\operatorname{Hom}_{{\bf CommAlg}_k}(\operatorname{Sym}_k(V\otimes C),B)=\prod_{x\in X}\operatorname{Hom}_k(C,B),
\end{equation*}
so by Yoneda lemma
\begin{equation*}
(T_k(V))_C=\operatorname{Sym}_k(V\otimes C)=k[X].
\end{equation*}
The argument easily generalizes to DG settings when $V$ is locally of finite dimension (\Cref{Def_LocallyFinite}).
\end{example}

We also collect some definitions of DG-manifolds in order to state \Cref{Proposition_HigherDRepAtPoint}.

\begin{definition}[{\cite[Section 2.2 and Section 2.5]{CK01}}]
A \textbf{DG scheme} is a pair $(X_0,\mathcal{O}_{X,\bullet})$, where $(X_0,\mathcal{O}_{X,0})$ is an ordinary scheme and $\mathcal{O}_{X,\bullet}$ is a sheaf of ($\mathbb{Z}_{\geq0}$-graded) commutative DG-algebras on $X_0$ such that each $\mathcal{O}_{X,i}$ is quasi-coherent over $\mathcal{O}_{X_0}$.
A DG scheme $X=(X_0,\mathcal{O}_{X,\bullet})$ is called \textbf{affine} if $X_0$ is affine.
In the affine case, we write $\pi_0(X):=\mathrm{Spec}~\mathrm{H}_0(\mathcal{O}_{X,\bullet})$ and identify $\pi_0(X)$ as a closed subscheme of $X_0$.

Given a DG scheme $X$ over $k$, a field extension $F\supseteq k$ and a closed $F$-point $x\in X_0$, we define the \textbf{DG tangent space} $(T_xX)_\bullet$ at $x$ to be the derivation complex
\begin{equation*}
(T_xX)_\bullet:=\underline{\mathrm{Hom}}_{\mathcal{O}_X}(\Omega^1_{\mathrm{comm}}(\mathcal{O}_{X}),\mathcal{O}_X)\otimes_{\mathcal{O}_X}F_x\cong\underline{\mathrm{Der}}_k(\mathcal{O}_{X},F_x),
\end{equation*}
where $F_x:=\mathcal{O}_{X,\bullet}/\mathfrak{m}_x$ as an $\mathcal{O}_{X,\bullet}$-module.
The homology groups of this complex are denoted $\pi_i(X,x):=\mathrm{H}_i(T_xX)$ and called the \textbf{derived tangent spaces} of $X$ at $x$.
\end{definition}

The following \Cref{Proposition_HigherDRepAtPoint} is a slight generalization of \cite[Proposition 7]{BFR14}, with a very similar proof.
This proposition will be an important step in the proof of \Cref{Thm_SmoothnessHigherVanishing}.

\begin{proposition}[c.f. {\cite[Proposition 7]{BFR14}}]\label{Proposition_HigherDRepAtPoint}
Let $F\supseteq k$ be a field extension of $k$.
Let $M$ be a finite-dimensional $k$-algebra and denote by $M_F$ the finite-dimensional algebra $M\otimes_kF$ over $F$.
Let $\rho:A\to M_F$ be a fixed representation of $A$ with coefficient $M_F$.
Then there is a canonical isomorphism
\begin{equation*}
\pi_i(\mathrm{DRep}_M(A),\rho)=\left\{\begin{matrix}\mathrm{Der}_k(A,M_F)&i=0\\ \mathrm{HH}^{i+1}(A,M_F)&i\geq1
\end{matrix}\right.
\end{equation*}
where $\mathrm{HH}^{i+1}(A,M_F)$ denotes the Hochschild cohomology and $M_F$ is regarded as a bimodule via $\rho:A\to M_F$.
\end{proposition}

\begin{remark}
We emphasize that the notation $\pi_i(X,x)$ is \emph{not defined} as the (derived) fiber $x^*\mathcal{O}_{X,\bullet}:=\mathcal{O}_{X,\bullet}\otimes_{\mathcal{O}_{X,0}}^{\mathbb{L}}F_x$.
By the universal property of the cotangent complex (for instance, \cite[Proposition 1.2.1.2]{TV08} in the affine case), there is a natural equivalence of derived mapping spaces
\begin{equation*}
\underline{\mathrm{Der}}_k(\mathcal{O}_{X},F_x)\simeq 
\operatorname{Map}_{\mathbf{Mod}_{\mathcal{O}_{X}}}(\mathbb{L}_{X/k},F_x),
\end{equation*}
and therefore
\begin{equation}\label{Eq_MappingSpTranslation}
\pi_i(X,x)\cong\mathrm{H}_i(\operatorname{Map}_{\mathbf{Mod}_{\mathcal{O}_{X}}}(\mathbb{L}_{X/k},F_x)).
\end{equation}
In this sense, $\pi_i(X,x)$ should be viewed as the algebraic invariant associated with the derived tangent space, although it is highly related to the derived fiber at the point $x$.
\end{remark}

\section{Formal smoothness test}\label{Sec_FormalSm}

In this section, we establish a formal smoothness test using derived representation schemes with finite-dimensional coefficients.
First, we recall the definition of noncommutative smoothness.

\begin{definition}[{\cite[Proposition 3.3]{CQ95},\cite{KR00}}]\label{Def_Smoothness}
An associative algebra $A$ is called \textbf{formally smooth} (or quasi-free) if any of the following equivalent conditions holds:
\begin{enumerate}
\item $A$ has cohomological dimension $\leq1$ with respect to Hochschild cohomology.
\item The universal bimodule of derivations $\Omega^1_{\mathrm{nc}}(A)$ is a projective bimodule.
\item $A$ satisfies the lifting property with respect to nilpotent extensions in $\bf{Alg}_k$, namely for every algebra homomorphism $A\to B/I$ where $I\subseteq B$ is a nilpotent ideal, there is an algebra homomorphism $\tilde{f}:A\to B$ inducing $f$.
\end{enumerate}
We say that $A$ is \textbf{smooth} if it is formally smooth and finitely generated.
\end{definition}

The next proposition shows that formal smoothness is reflected by derived representation schemes with arbitrary coefficients, in a manner analogous to the classical result for representation schemes with matrix coefficients.

\begin{proposition}[c.f. {\cite[Proposition 19.1.4]{G05}}]\label{Proposition_SmoothnessOfGeneralRepSchemes}
If $A$ is a formally smooth associative algebra, then $(A)_{M^*}$ is a formally smooth commutative algebra for any finite-dimensional algebra $M$.
\end{proposition}

We first establish the following lemma, which will be used in the proof of \Cref{Proposition_SmoothnessOfGeneralRepSchemes}.

\begin{lemma}\label{Lemma_PushDownNilpIdeals}
Let $M$ be a finite-dimensional algebra.
We regard it as a functor $\bm{\mathrm{CommAlg}}_k\to\bm{\mathrm{Alg}}_k$ defined by $B\mapsto M\otimes B$.
For any commutative $k$-algebra $B$ and a nilpotent ideal $I$, the two-sided ideal $M(I)$ is nilpotent, where
\begin{equation*}
M(I):=\mathrm{Ker}(M(B)\to M(B/I)).
\end{equation*}
\end{lemma}
\begin{proof}
First, every finite-dimensional algebra $M$ can be embedded into a matrix algebra $\mathrm{End}_k(M)$ (because every element $a\in M$ is a $k$-linear map $x\mapsto ax$), hence we have a diagram
\begin{center}
\begin{tikzcd}
0\arrow{r}{}&M(I)\arrow{r}{}\arrow[hook]{d}{}&M(B)\arrow{r}{}\arrow[hook]{d}{}&M(B/I)\arrow{r}{}\arrow[hook]{d}{}&0\\
0\arrow{r}{}&\mathrm{End}_k(M)(I)\arrow{r}{}&\mathrm{End}_k(M)(B)\arrow{r}{}&\mathrm{End}_k(M)(B/I)\arrow{r}{}&0
\end{tikzcd}
\end{center}
The map $M(I)\to\mathrm{End}_k(M)(I)$ is injective by the 5-lemma, so it suffices to show the result for $M=\mathrm{M}_d$, the matrix algebra.

This follows immediately from the structure of matrix multiplication, because
\begin{equation*}
\mathrm{M}_d(I)=\{d\times d\text{ matrices consisting of entries in }I\}.
\end{equation*}
\end{proof}

\begin{proof}[Proof of \Cref{Proposition_SmoothnessOfGeneralRepSchemes}]
To show that $(A)_{M^*}$ is formally smooth, it suffices to verify that for any nilpotent ideal $I$ of an associative algebra $B$, the map
\begin{equation}\label{Eq_SmoothnessCheck1}
\mathrm{Hom}_{\bm{\mathrm{CommAlg}}_k}((A)_{M^*},B)\to\mathrm{Hom}_{\bm{\mathrm{CommAlg}}_k}((A)_{M^*},B/I)
\end{equation}
is surjective.
Via adjunction (\ref{Eq_RepFunctorCoalgCoefficientAdj}), homomorphism (\ref{Eq_SmoothnessCheck1}) is identified with
\begin{equation}\label{Eq_SmoothnessCheck2}
\mathrm{Hom}_{\bm{\mathrm{Alg}}_k}(A,M(B))\to\mathrm{Hom}_{\bm{\mathrm{Alg}}_k}(A,M(B)/M(I)),
\end{equation}
and its surjectivity follows from the formal smoothness of $A$ and \Cref{Lemma_PushDownNilpIdeals}.
\end{proof}

Now we come to the theorem:

\begin{theorem}[{c.f. \cite[Theorem 21]{BFR14}}]\label{Thm_SmoothnessHigherVanishing}
Let $M$ be a finite-dimensional $k$-algebra.
Given an associative $k$-algebra $A$, if $A$ is formally smooth and finitely generated, then
\begin{equation*}
\mathrm{HR}_i(A,M)=0,\;\;\;i>0.
\end{equation*}
\end{theorem}

\begin{proof}[Proof of \Cref{Thm_SmoothnessHigherVanishing}]
Since $A$ is finitely generated, $A$ has a semifree resolution $q:Q\xrightarrow{\sim}A$ in ${\bf DGA}_k^{\geq0}$ such that $Q_0$ is a finitely generated free algebra.
Applying the functor $(-)_{M^*}$ to the resolution, we have a map
\begin{equation*}
(q)_{M^*}:(Q)_{M^*}\to(A)_{M^*}
\end{equation*}
of DG commutative algebras (by viewing $(A)_{M^*}$ as a DG algebra concentrated in degree $0$).

By \Cref{Proposition_HigherDRepAtPoint}, for any closed $F$-point $\rho\in\mathrm{Rep}_M(A)$,
\begin{equation*}
\pi_i(\mathrm{DRep}_M(A),\rho)\cong\mathrm{HH}^{i+1}(A,M_F),\;i\geq1.
\end{equation*}
Since $A$ is formally smooth, $\pi_i(\mathrm{DRep}_M(A),\rho)=0$ for $i\geq1$.
Consequently, $(q)_{M^*}$ induces isomorphisms $\tilde{f}_{\rho,i}:\pi_i(\mathrm{Rep}_M(A),\rho)\xrightarrow{\cong}\pi_i(\mathrm{DRep}_M(A),\rho)$ at each point $\rho$ and for each $i$.

Denote by $F_\rho$ the $(A)_{M^*}$-module (hence $(Q)_{M^*}$-module) corresponding to the point $\rho\in\mathrm{Rep}_M(A)$, with underlying field $F$.
By \Cref{Eq_MappingSpTranslation},
\begin{align*}
\pi_i(\mathrm{DRep}_M(A),\rho)&\cong\mathrm{H}_i(\operatorname{Map}_{\mathbf{Com}_{(Q)_{M^*}}}(\mathbb{L}_{(Q)_{M^*}},F_\rho))\\
&\cong\mathrm{H}_i(\operatorname{Map}_{\mathbf{Com}_{F_\rho}}(\mathbb{L}_{(Q)_{M^*}}\otimes^\mathbb{L}_{(Q)_{M^*}}F_\rho,F_\rho))\\
&\cong\mathrm{H}_i(\operatorname{Map}_{\mathbf{Com}_{F_\rho}}(\mathbb{L}_{(Q)_{M^*}}\otimes^\mathbb{L}_{(Q)_{M^*}}(A)_{M^*}\otimes^\mathbb{L}_{(A)_{M^*}}F_\rho,F_\rho))
\end{align*}
and similarly
\begin{equation*}
\pi_i(\mathrm{Rep}_M(A),\rho)\cong\mathrm{H}_i(\operatorname{Map}_{\mathbf{Com}_{F_\rho}}(\mathbb{L}_{(A)_{M^*}}\otimes^\mathbb{L}_{(A)_{M^*}}F_\rho,F_\rho)).
\end{equation*}
Thus under these identifications
\begin{equation}
\tilde{f}_{\rho}:\operatorname{Map}_{\mathbf{Com}_{F_\rho}}(\mathbb{L}_{(A)_{M^*}}\otimes^\mathbb{L}_{(A)_{M^*}}F_\rho,F_\rho)\to\operatorname{Map}_{\mathbf{Com}_{F_\rho}}(\mathbb{L}_{(Q)_{M^*}}\otimes^\mathbb{L}_{(Q)_{M^*}}(A)_{M^*}\otimes^\mathbb{L}_{(A)_{M^*}}F_\rho,F_\rho)
\end{equation}
is a quasi-isomorphism.

The functoriality of the identification (\ref{Eq_MappingSpTranslation}) implies that the map $\tilde{f}_{\rho}$ comes from the fundamental triangle of the cotangent complex.
In fact, we have the triangle in $D((A)_{M^*})$
\begin{equation}\label{Eq_FundamentalTriangle}
\mathbb{L}_{(Q)_{M^*}}\otimes^\mathbb{L}_{(Q)_{M^*}}(A)_{M^*}\xrightarrow{f}\mathbb{L}_{(A)_{M^*}}\to\mathbb{L}_{(A)_{M^*}/(Q)_{M^*}}\to
\end{equation}
where $\mathbb{L}_{(A)_{M^*}/(Q)_{M^*}}$ is the relative cotangent complex of $(q)_{M^*}$.
(\ref{Eq_FundamentalTriangle}) yields a triangle in $D(F)$
\begin{equation}\label{Eq_TriangleAtPoints}
\mathbb{L}_{(Q)_{M^*}}\otimes^\mathbb{L}_{(Q)_{M^*}}(A)_{M^*}\otimes^\mathbb{L}_{(A)_{M^*}}F_\rho\xrightarrow{f_\rho}\mathbb{L}_{(A)_{M^*}}\otimes^\mathbb{L}_{(A)_{M^*}}F_\rho\to\mathbb{L}_{(A)_{M^*}/(Q)_{M^*}}\otimes^\mathbb{L}_{(A)_{M^*}}F_\rho\to
\end{equation}
By the naturality of the identification (\ref{Eq_MappingSpTranslation}) via the map $(q)_{M^*}$, one checks $\tilde{f}_\rho=\operatorname{Map}_{\mathbf{Com}_{F_\rho}}(f_\rho,F_\rho)$.
Since $F_\rho=F$ is a field, the functor $\operatorname{Map}_{\mathbf{Com}_{F_\rho}}(-,F_\rho)$ reflects quasi-isomorphisms.
Hence $f_\rho$ is a quasi-isomorphism, and the triangle (\ref{Eq_TriangleAtPoints}) implies that $\mathbb{L}_{(A)_{M^*}/(Q)_{M^*}}\otimes^\mathbb{L}_{(A)_{M^*}}F_\rho\simeq0$ for each $\rho$.

\Cref{Example_FreeAlg} implies that $(Q_0)_{M^*}$ is an ordinary Noetherian commutative $k$-algebra, and so is its quotient $(A)_{M^*}$.
The derived pullback of $\mathbb{L}_{(A)_{M^*}/(Q)_{M^*}}$ to every point $\rho\in\mathrm{Rep}_M(A)$ vanishes.
By \cite[Lemma 2.11]{N92}, it follows that $\mathbb{L}_{(A)_{M^*}/(Q)_{M^*}}\simeq0$.
Hence, $(q)_{M^*}$ is a quasi-isomorphism by \cite[Corollary 3.2.17]{Lu04}.
\end{proof}

\begin{remark}
The proof of \cite[Theorem 21]{BFR14} requires a derived finiteness condition.
More precisely, $A$ is required to have a semifree resolution $q:Q\xrightarrow{\sim}A$ in ${\bf DGA}_k^{\geq0}$ whose underlying DG algebra is free of locally finite dimension (\Cref{Def_LocallyFinite}).
Here, \Cref{Thm_SmoothnessHigherVanishing} avoids the derived finiteness assumption by using different technical tools in place of \cite[Proposition 1.3]{K01}.
Although many examples of interest satisfy the derived finiteness assumption, the verification could be technically demanding.
In the appendix, we will discuss the relation between our approach and the main technical tool \cite[Proposition 1.3]{K01} used in the previous proof, and obtain a more refined comparison result.
\end{remark}

\begin{example}\label{Example_UniversalEnveloping}
Let $\mathfrak{g}$ be a semisimple Lie algebra over $\mathbb{C}$ of dimension $d\geq2$ and let $U\mathfrak{g}$ be its universal enveloping algebra.
It is known that
\begin{equation*}
\operatorname{HH}^n(U\mathfrak{g},N)\cong\operatorname{H}_{\mathrm{Lie}}^n(\mathfrak{g},N)
\end{equation*}
for any $\mathfrak{g}$-module $N$, where $\operatorname{H}_{\mathrm{Lie}}^n(\mathfrak{g},N)$ is the Lie cohomology of $\mathfrak{g}$ with coefficients in $N$.
By Poincar\'e duality, $\operatorname{H}_{\mathrm{Lie}}^d(\mathfrak{g},\mathbb{C})=\mathbb{C}$ where $\mathbb{C}$ denotes the trivial representation $\rho_0:U\mathfrak{g}\to\mathbb{C}$.
Thus by \Cref{Proposition_HigherDRepAtPoint},
\begin{equation*}
\pi_{d-1}(\mathrm{DRep}_{\mathbb{C}}(U\mathfrak{g}),\rho_0)=\mathrm{HH}^d(U\mathfrak{g},\operatorname{End}\mathbb{C})=\mathbb{C}.
\end{equation*}
Therefore the canonical map $\mathrm{Rep}_{\mathbb{C}}(U\mathfrak{g})\to\mathrm{DRep}_{\mathbb{C}}(U\mathfrak{g})$ is not a quasi-isomorphism, and in particular the higher representation homology of $U\mathfrak{g}$ does not vanish.
\emph{By \Cref{Thm_SmoothnessHigherVanishing}, $U\mathfrak{g}$ is not formally smooth.}

The calculation coincides with \cite[Example 2.3]{BKR13} in the case of $\mathfrak{sl}_2$.
In fact, the classical representation scheme $\mathrm{Rep}_{\mathbb{C}^m}(U\mathfrak{g})$ is smooth for all $m\geq1$ (see \cite[Theorem 3.4, Remark 4.7]{AGV16}), but the derived representation schemes detect the failure of smoothness of $U\mathfrak{g}$.

For semisimple Lie algebras, the obstruction to smoothness is detected entirely by the trivial representation.
In fact, by Weyl's Complete Reducibility Theorem, $N\cong\mathbb{C}^q\oplus P$, where $P$ is a $\mathfrak{g}$-module such that $P^{\mathfrak{g}}=0$.
Then $\operatorname{H}_{\mathrm{Lie}}^n(\mathfrak{g},N)=\operatorname{H}_{\mathrm{Lie}}^n(\mathfrak{g},\mathbb{C})^q$.
\end{example}

\section{Detecting non-smoothness using higher representation homology}\label{Sec_ExQ}

As shown in \cite[Section 6.2.2]{BFR14}, the quantum polynomial ring in two variables is not smooth, but all the higher representation homology groups with matrix coefficients vanish.
The next examples indicate that by allowing more general coefficients, representation homology detects the failure of smoothness.

\begin{example}\label{Example_QuantumPlane}
Let $k[x,y]_q$ be the quantum polynomial ring in two variables, i.e. $\displaystyle k[x,y]_q=\frac{k\langle x,y\rangle}{\langle xy-qyx\rangle}$,
where $q\in k^\times$ is not a root of unity.
It has a resolution given by the Shafarevich complex\footnote{See \cite{P05}} $k\langle x,y,t\mid dt=xy-qyx\rangle$.

Consider $\mathrm{HR}_*(k[x,y]_q,(\mathfrak{n}_3)_+)$, where $\mathfrak{n}_3$ is the (non-unital) algebra of $3\times3$ strictly upper triangular matrices over $k$, and $(\mathfrak{n}_3)_+$ is the augmentation of $\mathfrak{n}_3$ viewed as an object in ${\bf Alg}_k$.
A direct computation gives
\begin{align*}
\begin{split}
\begin{bmatrix}s&x&z\\ 0&s&y\\ 0&0&s\end{bmatrix}\begin{bmatrix}t&u&w\\ 0&t&v\\ 0&0&t\end{bmatrix}&-q\begin{bmatrix}t&u&w\\ 0&t&v\\ 0&0&t\end{bmatrix}\begin{bmatrix}s&x&z\\ 0&s&y\\ 0&0&s\end{bmatrix}\\ &=(1-q)\begin{bmatrix}st&su+tx&sw+tz+\frac{xv-quy}{1-q}\\ 0&st&sv+ty\\ 0&0&st\end{bmatrix}
\end{split}
\end{align*}
and so the following Koszul complex computes the representation homology (\cite[Theorem 2.8]{BKR13})
\begin{equation}\label{Eq_QPlaneKoszul}
k\left[s,t,x,y,z,u,v,w,t_1,t_2,t_3,t_4\middle| \begin{matrix}
dt_1=(1-q)st,\,dt_2=(1-q)(sx+tu),\,\\ dt_3=(1-q)(sy+tv),\\ dt_4=(1-q)(sw+tz+\frac{xv-quy}{1-q})\end{matrix}\right].
\end{equation}
%Using \verb|Macaulay2|, one can easily check that this Koszul complex has nontirvial homology group.
This Koszul complex has nontrivial homology.
In fact, observe that the ideal
\begin{equation*}
\left((1-q)st,\,(1-q)(sx+tu),\,(1-q)(sy+tv),\,(1-q)(sw+tz+\frac{xv-quy}{1-q})\right)
\end{equation*}
is contained in the ideal $(t,s,vx-quy)$, which has codimension at most $3$ by Krull's principal ideal theorem.
Thus the Koszul complex (\ref{Eq_QPlaneKoszul}) cannot be acyclic by \cite[Theorem 2.2]{L26}.
Therefore, $k[x,y]_q$ is not smooth.
\end{example}

\begin{example}\label{Example_JordanPlane}
Consider the Jordan plane $\displaystyle k[x,y]_J=\frac{k\langle x,y\rangle}{\langle yx-xy-x^2\rangle}$.
Its resolution can be given by $k\langle x,y,t\mid dt=yx-xy-x^2\rangle$, and we consider $\mathrm{HR}_*(k[x,y]_J,(\mathfrak{n}_3)_+)$ again.
Notice that
\begin{align*}
\begin{split}
\begin{bmatrix}t&u&w\\ 0&t&v\\ 0&0&t\end{bmatrix}\begin{bmatrix}s&x&z\\ 0&s&y\\ 0&0&s\end{bmatrix}&-\begin{bmatrix}s&x&z\\ 0&s&y\\ 0&0&s\end{bmatrix}\begin{bmatrix}t&u&w\\ 0&t&v\\ 0&0&t\end{bmatrix}-\begin{bmatrix}s&x&z\\ 0&s&y\\ 0&0&s\end{bmatrix}^2\\ &=\begin{bmatrix}-s^2&-2sx&-vx+uy-xy-2sz\\ 0&-s^2&-2sy\\ 0&0&-s^2\end{bmatrix},
\end{split}
\end{align*}
so the representation homology is computed via the Koszul complex (\cite[Theorem 2.8]{BKR13})
\begin{equation}\label{Eq_JPlaneKoszul}
k\left[s,t,x,y,z,u,v,w,t_1,t_2,t_3,t_4\middle| \begin{matrix}
dt_1=-s^2,\,dt_2=-2sx,\,dt_3=-2sy,\\ dt_4=-vx+uy-xy-2sz\end{matrix}\right].
\end{equation}
The same argument works here: the ideal
\begin{equation*}
\left(-s^2,\,-2sx,\,-2sy,\,-vx+uy-xy-2sz\right)
\end{equation*}
is contained in the ideal $(s,-vx+uy-xy)$, which has codimension at most $2$ by Krull's principal ideal theorem.
Thus the Koszul complex (\ref{Eq_JPlaneKoszul}) cannot be acyclic by the same argument as \Cref{Example_QuantumPlane}, and so $k[x,y]_J$ is not smooth.
\end{example}

Motivated by the preceding examples, we ask the following question:

\begin{question}\label{ques}
Let $A$ be a Noetherian finitely generated associative $k$-algebra.
Suppose that $\operatorname{Rep}_M(A)$ is not empty for any finite-dimensional algebra $M$.
If for \emph{every} finite-dimensional algebra $M$, the higher representation homology groups vanish, i.e.
\begin{equation*}
\mathrm{HR}_i(A,M)=0,\quad i>0,
\end{equation*}
is $A$ formally smooth?
\end{question}

%In derived algebraic geometry, general spaces are understood to be approximated by well-behaved ones. \Cref{Thm_SmoothnessHigherVanishing} suggests that noncommutative smooth spaces are already ``well-behaved," in the sense that no higher homological information is needed for such approximations---these spaces are approximated by themselves.
%The philosophy underlying \Cref{ques} is that the pathological behavior of a space should be detected by higher homotopical information with suitably chosen coefficients.
%These higher elements are expected to serve as the obstructions.
%A successful proof of \Cref{ques} would show that the vanishing of representation homology with arbitrary coefficients is equivalent to associative smoothness under very mild hypotheses.

\begin{remark}
The hypothesis that $\operatorname{Rep}_M(A)$ is not empty for any finite-dimensional algebra $M$ is made to exclude algebras that do not admit sufficiently many finite-dimensional representations, for instance, the Weyl algebra $A_1$.
It is likely stronger than necessary.
For example, $\operatorname{Rep}_{\operatorname{End}k}(\mathrm{M}_2(k))=\emptyset$ but we will show in \Cref{Prop_FinDimAlg} that \Cref{ques} has an affirmative answer for $\mathrm{M}_2(k)$.
However, the hypothesis is a useful formulation of having sufficiently many representations, and one can expect that even a proof under this strong hypothesis will lead to new conceptual insights, and many examples in practice satisfy this assumption.
One may expect the hypothesis to be weakened after a successful proof.
\end{remark}

\Cref{ques} asks whether finite-dimensional coefficient algebras form a conservative family of test objects for associative smoothness.
The examples above show that certain coefficients detect nonsmoothness invisible to the classical matrix-valued representation homology.
If the answer is positive, then formal smoothness would be characterized by the higher vanishing of representation homology with arbitrary finite-dimensional coefficients.
In this case, the derived Kontsevich–Rosenberg principle would provide a complete characterization of associative smoothness, not merely a conceptual justification.
%Such a result would provide a conceptual bridge between commutative and noncommutative geometry through the methods of derived algebraic geometry.

\begin{remark}
Formal smoothness is Morita invariant, whereas representation homology with finite-dimensional coefficients is sensitive to the algebra structure, and is not Morita invariant in general.
This contrast is already visible in the simplest Morita equivalence $k\sim\mathrm{M}_2(k)$: the groups $\mathrm{HR}_*(k,k)$ and $\mathrm{HR}_*(\mathrm{M}_2(k),k)$ do not agree---$\mathrm{HR}_0(k,k)=k$ but $\mathrm{HR}_0(\mathrm{M}_2(k),k)=0$.

Moreover, even allowing arbitrary finite-dimensional coefficients does not restore Morita invariance.
More precisely, there does not exist a finite-dimensional algebra $M$ such that $\operatorname{Rep}_{\operatorname{End}k}(k)\cong\operatorname{Rep}_M(\mathrm{M}_2(k))$.
Suppose on the contrary there is an $M$, and let $f:\mathrm{M}_2(k)\to M$ be an algebra homomorphism, then it is injective since $\mathrm{M}_2(k)$ is simple.
For a non-scalar invertible element $g\in\mathrm{M}_2(k)$, $x\mapsto f(gxg^{-1})$ is also a map of algebras.
$\operatorname{Rep}_{\operatorname{End}k}$ consists of only one point, so $x\mapsto f(gxg^{-1})$ is the same as $f$, hence
\begin{equation*}
f(x)=f(g)f(x)f(g)^{-1}
\end{equation*}
for all $x\in\mathrm{M}_2(k)$.
By the injectivity of $f$, $g$ is in the centralizer of $\mathrm{M}_2(k)$, which contradicts the fact that $g$ is non-scalar.

Since the classical representation scheme is recovered by degree-zero representation homology, the above examples show that representation homology with finite-dimensional coefficients is not Morita invariant in general.
\end{remark}

The following proposition gives an affirmative answer for finite-dimensional algebras:

\begin{proposition}\label{Prop_FinDimAlg}
Let $A$ be a finite-dimensional $k$-algebra, where $k$ is algebraically closed.
If $\operatorname{HR}_i(A,M)=0$ for all $i>0$ and each finite-dimensional $k$-algebra $M$, then $A$ is formally smooth.
\end{proposition}

\begin{proof}
Let $N$ be any finite-dimensional $A$-bimodule.
The trivial extension $M:=A\ltimes N$ is a finite-dimensional $k$-algebra, and there is a representation via the canonical inclusion $\rho_0:A\hookrightarrow A\ltimes N$, which is a $k$-point of $\operatorname{Rep}_{A\ltimes N}(A)$.
The higher vanishing hypothesis implies that the canonical map $\mathrm{DRep}_{A\ltimes N}(A)\to\mathrm{Rep}_{A\ltimes N}(A)$ is a quasi-isomorphism.
Hence $\pi_i\left(\operatorname{DRep}_{A\ltimes N}(A),\rho_0\right)=0$ for all $i$.
Therefore
\begin{equation*}
\operatorname{HH}^{i+1}(A,A\ltimes N)\cong\pi_i\left(\operatorname{DRep}_{A\ltimes N}(A),\rho_0\right)=0,
\end{equation*}
where the first equality is \Cref{Proposition_HigherDRepAtPoint}.

Under the $A$-bimodule structure induced by $\rho_0$, $A\ltimes N\cong A\oplus N$ as $A$-bimodules; therefore, by additivity of Hochschild cohomology,
\begin{equation*}
\operatorname{HH}^{i+1}(A,A)\oplus\operatorname{HH}^{i+1}(A,N)=\operatorname{HH}^{i+1}(A,A\ltimes N)=0
\end{equation*}
for all $i\geq1$.
Therefore, the $k$-vector space $\operatorname{HH}^n(A,N)=0$ for all $n\ge2$; in particular, $\operatorname{HH}^2(A,N)=0$ for every finite-dimensional bimodule $N$.
By \cite[Theorem 4.6]{AGV16}, $A$ is formally smooth.
\end{proof}

\begin{remark}
\Cref{Prop_FinDimAlg} shows that
\begin{equation*}
\operatorname{HH}^n(A,N)=0
\end{equation*}
for all $n>1$ and all finite-dimensional $A$-bimodules $N$.
The assumption that $k$ is algebraically closed is used only through \cite[Theorem 4.6]{AGV16}, which characterizes finite-dimensional formally smooth algebras by the vanishing of Hochschild cohomology with coefficients in finite-dimensional bimodules.
We have not investigated whether the algebraically closed hypothesis can be removed in the proof of \cite[Theorem 4.6]{AGV16}.
If an analogue of \cite[Theorem 4.6]{AGV16} over an arbitrary field is available, the same proof yields \Cref{Prop_FinDimAlg} without the algebraically closed hypothesis.
\end{remark}

It is also natural to ask whether analogous statements of \Cref{ques} can be developed for nonaffine noncommutative spaces and for geometric properties beyond smoothness.

\appendix

\section{A comparison criterion for DG schemes}

In this appendix, we reformulate the argument underlying \Cref{Thm_SmoothnessHigherVanishing} as a comparison criterion for DG schemes.
Although the ingredients are standard, the resulting criterion recovers the comparison theorem of Ciocan-Fontanine–Kapranov under substantially weaker hypotheses.
Since this discussion is orthogonal to the main theme of the paper, we place it in an appendix.

The starting point is the following comparison theorem of Kapranov, which plays a key role in the proof of \cite[Theorem 21]{BFR14}.

\begin{lemma}[{\cite[Proposition 1.3]{K01}}]\label{Lemma_DerSchComparison}
Let $f:(X_0,\mathcal{O}_{X,\bullet})\to(Y_0,\mathcal{O}_{Y,\bullet})$ be a morphism between DG schemes, where both $X=(X_0,\mathcal{O}_{X,\bullet})$ and $Y=(Y_0,\mathcal{O}_{Y,\bullet})$ satisfy the following properties\footnote{Such DG schemes are called DG-manifolds according to \cite{K01,CK01}.}:
\begin{enumerate}
\item The underlying schemes $(X_0,\mathcal{O}_{X,0})$ and $(Y_0,\mathcal{O}_{Y,0})$ are both ordinary smooth varieties.
\item There exist locally finite-dimensional vector spaces $E_\bullet$ and $F_\bullet$ such that $\mathcal{O}_{X,\bullet}\cong\operatorname{Sym}_{\mathcal{O}_{X,0}}(E_\bullet)$ and $\mathcal{O}_{Y,\bullet}\cong\operatorname{Sym}_{\mathcal{O}_{Y,0}}(F_\bullet)$.
\end{enumerate}
Then the following are equivalent:
\begin{enumerate}[label=(\alph*)]
\item $f$ is a quasi-isomorphism.
\item $\pi_0(f):\pi_0(X)\to\pi_0(Y)$ is an isomorphism, and for any field extension $F$ of $k$ and for any $F$-point $x\in X$, the induced maps $df_x:\pi_i(X,x)\to\pi_i(Y,f(x))$ are isomorphisms for all $i>0$.
\end{enumerate}
\end{lemma}

The following two lemmas are well known; we include proofs only for completeness.

\begin{lemma}[c.f. {\cite[Lemma 2.11]{N92}}]\label{Lemma_GlobalNeeman}
Let $X$ be a locally Noetherian scheme.
Given an object $C\in D(\mathrm{QCoh}(X))$, then $C\simeq0$ if and only if for any $x\in X$, $C\otimes^{\mathbb{L}}\kappa(x)\simeq0$, where $\kappa(x)$ is the residue field of the point $x$.
\end{lemma}
\begin{proof}
For any locally Noetherian scheme $X$, it is covered by affine Noetherian opens $\displaystyle X=\bigcup_{i\in\Lambda}X_i$.
Over each $X_i$, we know $C|_{X_i}\otimes^{\mathbb{L}}_{\mathcal{O}(X_i)}\kappa(x)\simeq0$ for all $x\in X_i$, hence $C|_{X_i}\simeq0$ by \cite[Lemma 2.11]{N92}.
Therefore $C\simeq0$.
\end{proof}

\begin{lemma}[c.f. {\cite[Corollary 3.2.17]{Lu04}}]\label{Lemma_GlobalLurie}
A morphism $f:X\to Y$ of DG-schemes (simplicial schemes) is an equivalence if and only if $f$ induces an isomorphism
$\pi_0(X)\to\pi_0(Y)$ and $\mathbb{L}_{X/Y}=0$.
\end{lemma}

\begin{proof}
The necessity is straightforward.
We argue by induction along the Postnikov tower and show that $f$ induces isomorphisms on all homotopy groups.
By the universal property of the cotangent complex, $\tau_{\leq1}\mathbb{L}_{Y/X}=0$ if and only if $f:X\to Y$ has the lifting property with respect to maps $T'\to T$ of $0$-truncated DG-schemes whose kernel of sheaves $\mathcal{I}$ is concentrated in degree $0$ and satisfies $\mathcal{I}^2=0$.
Similarly, $\tau_{\leq2}\mathbb{L}_{Y/X}=0$ if and only if $f$ has the lifting property against all maps whose kernel is concentrated in degrees $<2$.
Here, the square-zero condition is trivial by degree arithmetic.
This argument can be extended to higher degree, namely, $\tau_{\leq n}\mathbb{L}_{Y/X}=0$ if and only if $f$ has the lifting property against all maps whose kernel is concentrated in degrees $<n$.
Therefore, the full cotangent complex vanishes if and only if $f$ lifts against maps with kernel concentrated in degrees $<n$ for all $n$.
This means $f$ induces isomorphisms on all homotopy groups, which in particular implies that $f$ is a quasi-isomorphism.
\end{proof}

\begin{remark}
The proof in \cite[Corollary 3.2.17]{Lu04} is performed largely by showing that if $f:A\to B$ of simplicial algebras is $n$-connected, then $\mathbb{L}_{A/B}$ is $n+2$-connected.
The proof here is almost the same, although we use a more global argument. 
\end{remark}

\begin{proposition}\label{Prop__DerSchComparison}
Let $f:(X_0,\mathcal{O}_{X,\bullet})\to(Y_0,\mathcal{O}_{Y,\bullet})$ be a morphism between DG schemes, where $X_0$ is a locally Noetherian scheme.
Then the following are equivalent:
\begin{enumerate}[label=(\alph*)]
\item $f$ is a quasi-isomorphism.
\item $\pi_0(f):\pi_0(X)\to\pi_0(Y)$ is an isomorphism, and for any field extension $F$ of $k$ and for any $F$-point $x\in X$, the induced maps $df_x:\pi_i(X,x)\to\pi_i(Y,f(x))$ are isomorphisms for all $i>0$.
\end{enumerate}
\end{proposition}
\begin{proof}
The direction (a)$\Longrightarrow$(b) is straightforward, so we only show (b)$\Longrightarrow$(a) here.

We have the distinguished triangle in $D(X)$ (hence in $D(X_0)$)
\begin{equation*}
\mathbb{L}f^*\mathbb{L}_{Y}\xrightarrow{\varphi}\mathbb{L}_{X}\to\mathbb{L}_{X/Y}\to,
\end{equation*}
where $\varphi$ is induced by $f:X\to Y$.
Given any point $x\in X_0$, applying the functor $-\otimes^\mathbb{L}_{\mathcal{O}_X}\kappa(x)$ yields another triangle
\begin{equation*}
\mathbb{L}f^*\mathbb{L}_{Y}\otimes^\mathbb{L}_{\mathcal{O}_X}\kappa(x)\xrightarrow{\varphi_x}\mathbb{L}_{X}\otimes^\mathbb{L}_{\mathcal{O}_X}\kappa(x)\to\mathbb{L}_{X/Y}\otimes^\mathbb{L}_{\mathcal{O}_X}\kappa(x)\to,
\end{equation*}
and then applying the functor $\operatorname{Map}_{\kappa(x)}(-,\kappa(x))$ and derived tensor-hom adjunction gives the following triangle
\begin{equation}\label{Eq_ComparisonTriangle}
\operatorname{Map}_{\mathbf{Mod}_{\mathcal{O}_{X}}}(\mathbb{L}_{X/Y},\kappa(x))\to\operatorname{Map}_{\mathbf{Mod}_{\mathcal{O}_{X}}}(\mathbb{L}_{X},\kappa(x))\xrightarrow{\varphi_x^*}\operatorname{Map}_{\mathbf{Mod}_{\mathcal{O}_{X}}}(\mathbb{L}f^*\mathbb{L}_{Y},\kappa(x))\to.
\end{equation}
By \Cref{Eq_MappingSpTranslation}, $\pi_i(X,x)\cong\mathrm{H}_i(\operatorname{Map}_{\mathbf{Mod}_{\mathcal{O}_{X}}}(\mathbb{L}_{X/k},\kappa(x)))$, and the map $\pi_i(\varphi_x^*)$ in \Cref{Eq_ComparisonTriangle} can be identified with $\pi_i(f)$ by functoriality.
Hence $\operatorname{Map}_{\mathbf{Mod}_{\mathcal{O}_{X}}}(\mathbb{L}_{X/Y},\kappa(x))\simeq0$ by assumption (b).
Since the functor $\operatorname{Map}_{\kappa(x)}(-,\kappa(x))$ is conservative, $\mathbb{L}_{X/Y}\otimes^\mathbb{L}_{\mathcal{O}_X}\kappa(x)\simeq0$.
\Cref{Lemma_GlobalNeeman} implies that $\mathbb{L}_{X/Y}\simeq0$, and $f$ is a quasi-isomorphism by \Cref{Lemma_GlobalLurie}.
\end{proof}

\begin{corollary}
\Cref{Lemma_DerSchComparison} holds.
\end{corollary}
\begin{proof}
Since $X_0$ is a smooth variety, it is a quasi-compact, quasi-separated, locally Noetherian scheme.
Hence, by \Cref{Prop__DerSchComparison}, the two conditions in \Cref{Lemma_DerSchComparison} are equivalent.
\end{proof}

%\Cref{Prop__DerSchComparison} holds true under the substantially weaker assumption that the underlying classical scheme is locally Noetherian.

\begin{remark}
We can also prove the comparison result under other structural or finiteness assumptions.
For instance, \cite[Proposition 3.2.18]{Lu04} identifies finite presentation with the perfection of $\mathbb{L}_{X/Y}$ together with a condition on $\pi_0$.
Separately, a result by Hopkins (also see \cite{N92}) says that the fiber test is sufficient for a perfect complex without the Noetherian assumption.
However, these do not seem to help loosen the hypothesis in \Cref{Thm_SmoothnessHigherVanishing}; hence we will not discuss them.
\end{remark}

\end{document}